\documentclass{article}

\usepackage{amsmath}
\usepackage{amsthm}
\usepackage{amssymb}
\usepackage{hyperref} 
\usepackage{color}
\usepackage{verbatim}

\newtheorem{theorem}{Theorem}
\newtheorem{lemma}[theorem]{Lemma}
\newtheorem{corollary}[theorem]{Corollary}

\theoremstyle{definition}

\title{\AE quitas: Two-Player Counterfeit Coin Games}
\author{Kyle Burke, Tanya Khovanova, Joshua Lee,\\ Richard J. Nowakowski, Amelia Rowland, Craig Tennenhouse}

\begin{document}
\maketitle

\begin{abstract}
We discuss games involving a counterfeit coin. Given one counterfeit coin among a number of otherwise identical coins, two players with full knowledge of the fake coin take turns weighing coins on a two-pan scale, under the condition that on every turn they reveal some information to an Observer with limited prior knowledge about the coins. We study several games depending on the types of the counterfeit coin, the end-game condition, and the availability of extra genuine coins. We present the winning positions and the Grundy numbers for these games.
\end{abstract}

\section{Introduction}

The problem of finding the counterfeit coin in a collection of indistinguishable coins has been extensively studied. See \cite{GuyN} for a dated but useful overview of the different aspects of the problem. 
Inspired by ideas of cryptography, new coin problems where the coins have a lawyer and want to limit the information that can be revealed during weighing were recently introduced \cite{DK}.

We are turning these new ideas into games. We study several games with the same initial set-up.

\begin{quote}
There are $N$ identically-looking coins, where one is counterfeit and the rest are real. All real coins weigh the same. There is a balance scale with two pans. During a weighing, the same number of coins is put on each pan. 

There is an Observer who wants to find the counterfeit coin. The Observer is infinitely intelligent and watches all the weighings.

Two players play the following game. The players have full information about coins and take turns weighing coins.  Each move consists of one weighing, given that the Observer gains some new information.
\end{quote}

We study several variations of this game depending on the types of the counterfeit coin, the end-game condition, and the availability of extra genuine coins.

In most coin problems, the counterfeit coin is assumed to be lighter for an obvious reason: in real life counterfeit coins are lighter. Mathematicians like symmetry, however, and consider the cases of the coin being lighter or heavier as equivalent. For many reasons, which will be revealed later, it is useful to combine lighter and heavier cases as follows:

Suppose we can split our coins into two groups. The Observer doesn't know where the counterfeit coin is, but she knows that if the counterfeit coin is in the first group it must be lighter, and if it is in the second group it must be heavier. 

We are ready for some definitions. 

We define a \textit{destiny} of a coin to be heavy or light. We consider a game, which we denote a $D_{l,h}$ game, where $l$ coins have a lighter destiny and $h$ coins have a heavier destiny. The total number of coins is $N=l+h$. If the counterfeit coin ends up being among the $l$ coins destined to be lighter, the counterfeit coin is lighter and vice versa.

The particular case when all the coins are destined to be light corresponds to the most famous classical case: there are several coins, and the counterfeit coin is known to be lighter than the others.
In our notation, it is a $D_{N,0}$ game. 

Now we are ready to describe the cases we study.

\begin{enumerate}
\item \textbf{The case of the destined coins.} The Observer knows about every coin's destiny and wants to \textit{find} the counterfeit coin. Automatically, as soon as the coin is found, the Observer would know whether it is heavier or lighter.
\item \textbf{The case of the unknown coins.} The Observer doesn't know whether the fake coin is heavier or lighter. There are two subcases: the Observer just wants to \textit{find} the coin, or the Observer wants not only to find it but also to \textit{identify} whether it is heavier or lighter.
\end{enumerate}

We promised to reveal another reason why the $D_{l,h}$ game is important other than that it is a generalization of $D_{N,0}$. When the counterfeit coin could be either heavier or lighter, and there is an unbalanced weighing, all the coins on the scale acquire a destiny. If the counterfeit coin is on the lighter pan, it has to be lighter and vice versa.

We call a coin \textit{excluded} if it is proven to be real after some weighings. These are the coins that were:
\begin{itemize}
    \item at least once on the scale when the weighing was balanced, or
    \item coins with a lighter destiny that were on a heavier pan, or
    \item coins with a heavier destiny that were on a lighter pan, or
    \item coins that were not on the scale during an imbalance.
\end{itemize}

We call a coin \textit{unknown} if it is not excluded and not destined. That means this coin theoretically can be counterfeit and either lighter or heavier.

\textbf{End-game.} We study different end-game conditions. The names for these come straight from the combinatorial game theory terms \emph{normal play} and \emph{mis\`ere play}.

\begin{enumerate}
\item The game ends when the Observer reaches her goal of finding the information about the counterfeit coin. The player who reveals the last piece of necessary info wins. We call this game the \textit{normal play}.
\item The player that reveals the full information to the Observer loses. That means the game ends, when on the next turn, the full information is guaranteed to be revealed. We call this game the \textit{mis\`ere play}.
\end{enumerate}

\textbf{Extra coin.} It is also convenient to have another variation:

\begin{itemize}
\item At the beginning of the game, an extra coin is available that the Observer and both players know is not counterfeit.  We include a superscripted $+$ for these games, so a normal play game with $N=l+h$ destined coins and one extra coin is denoted $D^+_{l, h}$.
\end{itemize}

The extra coin will be important. It simplifies the calculations. In addition, as soon as some coins are excluded, this situation arises automatically.

It might be surprising, but the case of the lighter coin and the case of the destined coins are almost identical. That is, the Grundy numbers and $\mathcal{P}$-positions depend on the total number of coins with few exceptions. This is covered in  Section~\ref{sec:li}, where we provide the result of the investigation of the case of the destined coins. 

In Section~\ref{sec:hlf} we investigate the case of the unknown coins. We find that in the presence of an extra coin, the find and find-identify cases are almost identical. If there is no extra coin, then only for mis\`ere play the cases of find and find-identify are almost identical.

\subsection{Sprague-Grundy Theory for Normal and Mis\`ere play, Impartial Games}\label{sec:nimvalues}

\textit{Combinatorial games} are played by two players who move alternately; 
the game finishes after a finite sequence of moves regardless of the order of play;
there is perfect information; and there are no chance devices. The two players are called \textit{Left} and \textit{Right}.
The \textit{Fundamental Theorem of Combinatorial Games} (see \cite{LiP}[Theorem 2.1]) gives that there are four outcome classes: 
 $\mathcal{L}$---Left can force a win regardless of moving first or second; 
  $\mathcal{R}$---Right can force a win regardless of moving first or second; 
   $\mathcal{N}$---the next player to play can force a win regardless of whether it is Left or Right; 
   and  $\mathcal{P}$---the next player to play cannot force a win regardless of whether it is Left or Right.

Let $G$ be a position (in a game). The \textit{Left options} of $G$ are those positions that Left can move to (in one move) and 
the \textit{Right options} are defined analogously. Let $G^\mathcal{L}$ and $G^\mathcal{R}$ be the sets of Left and Right 
options of $G$. Now $G$ can be identified with the sets of options, written 
$G=\{ G^\mathcal{L}\mid  G^\mathcal{R}\}$.
In this paper, we consider \textit{impartial} games, that is, the Left and Right options 
for any possible position of the game are identical---the two players always have the 
same options. As a consequence, the only outcome classes are $\mathcal{N}$ and $\mathcal{P}$. 
Since $G^\mathcal{L}=G^\mathcal{R}$ we need only write them once, as in $G=\{H\mid H \mbox{ is an option of $G$}\}$.

The minimum excluded value, \textit{mex}, of a set $S$ is the least non-negative integer not included in $S$. 

For normal play, recursively, the \textit{Grundy value}\footnote{The theory of impartial games was developed before that of partizan games, of which impartial games are a subset. The correspondence between them is: if $\mathcal{G}(G) =n$ then in partizan games
 the value of $G$, is denoted $*n$.}, of an impartial game $H$ 
 is given by 
 $$\mathcal{G}(H)=\mbox{mex}\{ \mathcal{G}(H')\mid H' \mbox{ is an option of } H\}.$$
Thus, if a game has no options, it has nim-value 0. There is a link between nim-values and outcomes:
$$\mathcal{G}(G) = 0 \mbox{ if and only if $G$ is a $\mathcal{P}$-position}.$$
The \textit{disjunctive sum} of two positions, $G$ and $H$, written $G+H$ is the position in which a player plays in either $G$ or $H$ but not both. 
 Let $p$ and $q$ be  non-negative integers then 
  $p\oplus q$ signifies the \textit{nim-sum} or \textit{exclusive or} of $p$ and $q$. It is obtained 
  by writing $p$ and $q$ in binary and adding in binary without carrying. 
  The value of the disjunctive sum $F + H$ is given by 
  $$\mathcal{G}(F+H) = \mathcal{G}(F)\oplus \mathcal{G}(H).$$ 

For mis\`ere play, the Grundy values are defined the same as normal play: 
 $$\mathcal{G}(H)=\mbox{mex}\{ \mathcal{G}(H')\mid H' \mbox{ is an option of } H\},$$
except $\mathcal{G}(\emptyset)=1$.
It is still true that $$\mathcal{G}(G) = 0 \mbox{ if and only if $G$ is a $\mathcal{P}$-position}$$
but it is no longer true that $\mathcal{G}(G+H)=\mathcal{G}(G)\oplus\mathcal{G}(H)$.
Mis\`ere play is trickier than normal play except when the game is equivalent to playing Nim, as
is the case in this game,
and an extension of Grundy values, called the \textit{genus}, is required.

 Under both winning conventions, 
 the  Grundy value is also called the nim value since if $\mathcal{G}(G)=k$, then playing in $G$ is equivalent to playing in a Nim heap of size $k$. The context will always ensure there is no confusion between the two functions.

\section{The Case of the Destined Coins}\label{sec:li}

There are $l$ coins that are destined to be lighter and $h$ coins that are destined to be heavier, where $l+h=N$. We denote  this game as $D_{l,h}$ for normal play and $DM_{l,h}$ for mis\`ere play. When $l=N$ we are looking for a lighter counterfeit coin. When $h = N$, we are looking for a heavier counterfeit coin.

Initially any of the coins could be counterfeit. Each new weighing excludes some coins from the list of potentially counterfeit coins. If a weighing is a balance, then all the coins on the scale must be real, while coins off the scale keep their destinies. If a weighing is an imbalance, then the lighter-destined coins that are on the heavier pan, the heavier-destined coins on the lighter pan, and the coins that are not on the scale must be real, while other coins keep their destinies.

\textbf{Example $N=2$.} If $l=h=1$, the only possible weighing doesn't reveal any additional information. In the future we exclude this case from our consideration, as it's impossible to determine the counterfeit coin. If $l=N=2$, the only possible weighing reveals the counterfeit coin.

\textbf{Example $N=3$.} We can only weigh one coin against one coin. These two coins must share a destiny to reveal information. In any case, after this weighing the Observer would know the identity of the counterfeit coin.

\subsection{Extra coin}

We start with a simpler case when there is an extra real coin. We denote this game as $D_{l,h}^+$ for normal play and $DM_{l,h}^+$ for mis\`ere play.

\begin{lemma}
The game $D_{l,h}^+$ is equivalent to the normal play of the game of Nim with one pile of size $l+h-1$. The game $DM_{l,h}^+$ is equivalent to the mis\`ere play of the game of Nim with one pile of size $l+h-1$.
\end{lemma}

\begin{proof}
In one move, a player can exclude any number of coins up to $N-1$. The player can pick real coins that he wants to exclude, and if the number is odd, add the extra real coin. Then the player puts these coins on the scale and demonstrates that the pans balance.
\end{proof}

In the normal play with an extra coin, the first player wins in one move by excluding all the real coins. In the mis\`ere play the first player wins by excluding all but two coins. The lemma also provides the Grundy values for the game:
\[ \mathcal{G}(D_{l,h}^+) = l+h-1 = N-1 \quad \text{ and } \quad  \mathcal{G}(DM_{l,h}^+) = l+h-2= N-2.\]

\subsection{No extra coin}

Let us go back to the game without the extra coin.

We need to study the first weighing. After the first weighing, some of the coins get excluded, and we are in the situation with extra real coins.

\begin{lemma}
The following list describes the total number of coins that can stay destined after the first weighing from the starting position in the game $D_{l,h}$:
\begin{itemize}
    \item For any $N=l+h$: any number of destined coins that is the same parity as $N$.
    \item For $N=2l$: any number of destined coins up to $N-2$.
    \item For $N \neq 2l$: any number of destined coins that is less than $N$ and not more than $\lfloor N/2\rfloor + \min\{l,h\}$ except:
    \item For even $N$ and  $\min\{l,h\} =1 $: we can't reach 1 destined coin.
\end{itemize} 
\end{lemma}

\begin{proof}
We can pick any even number of real coins and put them on the scale to balance, thus excluding them. This way we can reach any number of destined coins of the same parity as $N$.

Suppose $N=2l$. We already know how to get to any even number of destined coins. To get to an odd number $2i-1$, we can put $i$ coins with lighter destiny on the left pan. On the right pan we put $i-1$ coins with heavier destiny and one coin with lighter destiny. This weighing is possible if $i< l$. If $i=l$ we do not have enough $l$-coins.

Suppose $N \neq 2l$. Without loss of generality assume that $l > h$. In our weighing we put $i$ coins with lighter destiny on the left pan and on the right pan we put $j$ coins with heavier destiny and $i-j$ coins with lighter destiny. After the weighing the number of destined coins is $i+j$. This weighing is possible if $i \leq \lfloor N/2\rfloor$, and $j \leq h$.

Looking at the smallest number we can reach, we see that if our fake coin shares destiny with another coin, we can reach one destined coin by comparing these two coins. If it doesn't share, the only way to reach one destined coin is to exclude all the coins but the fake one. That means the number of other coins must be even, and we can balance all the other coins against each other.
\end{proof}

To calculate the Grundy values we need to know the smallest number of destined coins that is unreachable in one move.

\begin{corollary}
The smallest number of destined coins that is unreachable in one move is the following:
\begin{itemize}
\item For $N=2l$, it is $N-1$;
\item Otherwise, it is $\lfloor N/2\rfloor + \min\{l,h\}$ plus 1 or 2, whichever provides a different parity from $N$, except:
\item If $N$ is even and $\min\{l,h\} = 1$ we can't reach 1.
\end{itemize}
\end{corollary}

After the first weighing we have at least one excluded coin. Also the Grundy values for the case when there is an extra coin are consecutive values. Therefore, the Grundy value of our game equals the Grundy value of the game with one extra coin and the number of destined coins being the smallest number we can't reach.

Now we describe $\mathcal{P}$-positions. Remember that the case $l=1$, $N=2$ is not playable.

\begin{corollary}
In normal play the destined game has $\mathcal{P}$-position $N=1$. In addition, the $\mathcal{P}$-positions correspond to even $N$, when $l=1$ or $h=1$.

In mis\`ere play, the $\mathcal{P}$-positions are $N=2$ or $N=3$, the rest are $\mathcal{N}$-positions.
\end{corollary}

\begin{proof}
If $N>1$ in normal play the first player wins by comparing the counterfeit coin with any real coin of the same destiny. If the counterfeit coin is the only coin of the same destiny and the number of coins of the opposite destiny is even, then the player can exclude all other coins. Otherwise, the counterfeit coin is the only coin, and it's a $\mathcal{P}$-position.

In mis\`ere play if there are coins of different destinies, the first player wins by weighing the counterfeit coin against a coin of opposite destiny. If coins have the same destinies and $N>3$, the first player wins by comparing two coins against two coins and putting the counterfeit coin among them.
\end{proof}

Here are important examples.

\textbf{Example.} The classical coin-weighing problem has one counterfeit coin that is lighter. In honor of this problem we consider the case when $l=N$. In this case $\min\{l,h\} = 0$, and the smallest number of destined coins we can't reach in one weighing is the smallest number exceeding $\lfloor N/2\rfloor$ that differs in parity from $N$. The Grundy values for normal play and mis\`ere play are presented in Table~\ref{tab:destinedGrundy}.

\begin{table}[h!]
  \centering
\begin{tabular}{|c|c|c|c|}
\hline
  $N$   & can't reach & $\mathcal{G}(D_{N,0})$ & $\mathcal{G}(DM_{N,0})$ \\
 \hline
 $N = 4k$ & $2k+1$  & $2k$ & $2k-1$\\
  $N = 4k+1$ & $2k+2$  & $2k+1$ & $2k$\\
   $N = 4k+2$ & $2k+3$  & $2k+2$ & $2k+1$\\
    $N = 4k+3$ & $2k+2$  & $2k+1$ & $2k$\\
 
 \hline
\end{tabular}
  \caption{Grundy values, destined coins all lighter}
  \label{tab:destinedGrundy}
\end{table}

\textbf{Example.} In the next section we need another example when $l = h$. The Grundy values for normal play and mis\`ere play are presented in Table~\ref{tab:destinedGrundylh}.

\begin{table}[h!]
  \centering
\begin{tabular}{|c|c|c|c|}
\hline
  $N$   & can't reach & $\mathcal{G}(D_{l,l})$ & $\mathcal{G}(DM_{l,l})$ \\
 \hline
 $N = 2l$ & $2l-1$  & $2l-2$ & $2l-3$\\
 \hline
\end{tabular}
  \caption{Grundy values, destined coins $l=h$}
  \label{tab:destinedGrundylh}
\end{table}

\section{The Case of the Unknown Coins}\label{sec:hlf}

In this variant of the game, the counterfeit coin can be heavier or lighter, and the Observer either needs to find it or to find and identify it. Given $N$ unknown coins at the beginning, we denote the find game $F_N$, and the find-identify game $FI_N$ for normal play. The mis\`ere play is denoted $FM_N$ and $FIM_N$ correspondingly.

If a weighing is a balance, it excludes the coins on the scale. If a weighing is an imbalance, it excludes the coins that are not on the scale. Moreover, the coins on the scale acquire a destiny. 

There cannot be destined and unknown coins at the same time. In addition, as soon as an imbalance happens and the game turns to destined coins, the game stays as the game of destined coins. Moreover, as soon as destined coins appear, the find and find-identify games become equivalent. This means the Grundy values for the position after the imbalance are the same as Grundy values for the game of destined coins.

\subsection{Extra coin}

As before we first consider the case when there is an extra coin known to be real.  We denote the find game with the extra coin $F_N^+$, and the identify game as $FI_N^+$ for normal play, and $FM_N^+$ and $FIM_N^+$ for mis\`ere play.

Consider the first move.

\begin{lemma}
From the starting position, there is a move to any positive number of excluded coins not exceeding $N-1$ and the rest are unknown, and to any positive number of destined coins not exceeding $N$. There are no other moves.
\end{lemma}

\begin{proof}
If the first move is a balance with $m$ coins on each pan, we exclude $2m$ coins on both pans, the rest staying unknown. If we use the extra coin on the scale, we can exclude any odd number of coins.

If the first move is an imbalance with $m$ coins on each pan, the coins on the lighter pan get the lighter destiny, the coins on the heavier pan get the heavier destiny. If we do not put the extra coin on the scale, we get an even number of destined coins. If we put the extra coin on the scale, we get an odd number of destined coins.
\end{proof}

This allows us to find the Grundy values recursively. If we get to the destined coins at some point, then the find and find-identify games are equivalent. Indeed, if the coins have destinies, as soon as the counterfeit coin is found, its destiny provides the identification.

Now we look at the initial values for the recursion.

\textbf{One unknown coin and one extra real coin.} One unknown coin is a terminal position in games $F^+$ and $FIM^+$, one move away from a terminal position in $FI^+$, and impossible in $FM^+$. That means, \[\mathcal{G}(F_1^+)= 0 \quad \mathcal{G}(FI_1^+) = 1 \quad \mathcal{G}(FIM_1^+) = 1.\]

\textbf{Two unknown coins and one extra real coin.} There are three different moves that provide more information:
\begin{enumerate}
\item compare the counterfeit coin to the extra coin,
\item compare the real unknown coin to the extra coin,
\item compare two unknown coins to each other.
\end{enumerate}

The first move leads to finding and identifying the fake coin. The second move leads to finding but not identifying the fake coin. The third move leads to two destined coins with one coin having the lighter destiny and the other coin having the heavier destiny.

In Table~\ref{tab:2plus} we show Grundy values after each move for each game.

\begin{table}[h!]
  \centering
  \begin{tabular}{|c|c|c|c|c|}
  \hline
move & $F_2^+$   & $FI_2^+$ & $FM_2^+$ & $FIM_2^+$ \\
 \hline
1 & 0 & 0 & N/A & N/A    \\
2 & 0 & 1 & N/A     &  $1$  \\
3 & 1 & 1 & $1$ &       $1$   \\
\hline
  \end{tabular}
  \caption{Grundy values for moves from two unknown coins, extra coin}
  \label{tab:2plus}
\end{table}

Therefore,

\[\mathcal{G}(F_2^+)= 2 \quad \mathcal{G}(FI_2^+) = 2 \quad \mathcal{G}(FM_2^+) = 1 \quad \mathcal{G}(FIM_2^+)= 1.\]

Now we can describe all Grundy values.

\begin{lemma}
For $N > 1$, the games $F_N^+$ and $FI_N^+$ are equivalent to the normal play of the game of Nim
with one pile of size $N$. The games
$FM_N^+$ and $FIM_N^+$ are equivalent to the mis\`ere play
of the game of Nim with one pile of size $N$.
\end{lemma}

\begin{proof}
We use induction. The base case was covered in the examples above.

Let us look at the normal play. If we start with $N$ unknown coins, there exist moves to any number of destined coins up to $N$. That means there are moves to positions with Grundy values from 0 to $N-1$ inclusive. By the inductive hypothesis, all other moves to a smaller number of unknown coins have Grundy values in the same range. Therefore, the minimal excludant is $N$. A similar argument works for mis\`ere play.
\end{proof}

The Grundy values are combined in Table~\ref{tab:uG}. Notice that starting from $N=2$, the find and find-identify games are equivalent. If we allow only imbalances in our moves, the Grundy values won't change. The game with only imbalances is the same as the destined coin case with one extra move: the first move that assigns destinies.

\begin{table}[h!]
  \centering
  \begin{tabular}{|c|c|c|c|}
  \hline
 $F_N^+$   & $FI_N^+$ & $FM_N^+$ & $FIM_N^+$ \\
 \hline
 $N$   & $N$        &     $N-1$      &       $N-1$     \\
\hline
  \end{tabular}
  \caption{Grundy values: unknown case, extra coin}
  \label{tab:uG}
\end{table}

\begin{corollary}
For $N > 1$, the game can be won in one move.
\end{corollary}

\begin{proof}
In normal play, the player compares the counterfeit coin against the extra coin. In mis\`ere play, the player compares the counterfeit coin against another unknown coin.
\end{proof}

\subsection{No extra coin}

Consider the first move.

\begin{lemma}
From the starting position, there is a move to any even number of excluded coins not exceeding $N-2$, and the rest are unknown, and to any even number of destined coins together with the rest of the coins being excluded. There are no other moves.
\end{lemma}

\begin{proof}
If the first move is a balance with $m$ coins on each pan, we exclude $2m$ coins on both pans, the rest staying unknown. If the first move is an imbalance with $m$ coins on each pan, the $2m$ coins on the scale acquire a destiny, and the rest of the coins are excluded.
\end{proof}

That means after the first move, we either get an extra coin or we get destined coins. In either case we know the Grundy values of all these positions. 

We consider the case when the total number of coins is even or odd separately. 

The Grundy value for resulting positions for the even case, $N =2k$, are presented in Table~\ref{tab:evenU}. Note that when we put all the coins on the scale, we do not have an extra real coin after that weighing.

\begin{table}[h!]
  \centering
  \begin{tabular}{|c|c|c|c|c|}
  \hline
move & $F_{2k}$   & $FI_{2k}$ & $FM_{2k}$ & $FIM_{2k}$ \\
 \hline
$2i$ unknown, $0 < 2i <N$ & $2i$ & $2i$ & $2i-1$ & $2i-1$    \\
$2i$ destined, $0 < 2i <N$  & $2i-1$ & $2i-1$ & $2i-2$     &  $2i-2$  \\
$2k$ destined & $2k-2$ & $2k-2$ & $2k-3$ &       $2k-3$   \\
\hline
  \end{tabular}
  \caption{$N=2k$ moves}
  \label{tab:evenU}
\end{table}

Applying the minimal excludant function to each column, we get the Grundy numbers:

\begin{lemma}
For $N = 2k$, the Grundy values are presented by Table~\ref{tab:evenUG}.
\begin{table}[h!]
  \centering
  \begin{tabular}{|c|c|c|c|}
  \hline
$F_{N}$   & $FI_{N}$ & $FM_{N}$ & $FIM_{N}$ \\
\hline
$0$ & $0$ & $N-2$ &       $N-2$   \\
\hline
  \end{tabular}
  \caption{Grundy values $N =2k$}
  \label{tab:evenUG}
\end{table}
\end{lemma}

The Grundy value for resulting positions for the odd case, $N =2k+1$, are presented in Table~\ref{tab:oddU}. Note that whatever the first weighing, we always get an extra real coin. Also, one unknown coin has a separate formula.

\begin{table}[h!]
  \centering
  \begin{tabular}{|c|c|c|c|c|}
  \hline
move & $F_{2k+1}$   & $FI_{2k+1}$ & $FM_{2k+1}$ & $FIM_{2k+1}$ \\
 \hline
$2i+1$ unknown $1 < 2i+1 <N$ & $2i+1$ & $2i+1$ & $2i$ & $2i$    \\
$1$ unknown & $0$ & $1$ & N/A & $1$   \\
$2i$ destined, $0 < 2i <N$  & $2i-1$ & $2i-1$ & $2i-2$     &  $2i-2$  \\
\hline
  \end{tabular}
  \caption{$N=2k$ moves}
  \label{tab:oddU}
\end{table}

Applying the minimal excludant function to each column, we get the Grundy numbers:

\begin{lemma}
For $N = 2k+1$, the
Grundy values are presented by Table~\ref{tab:oddUG}.
\begin{table}[h!]
  \centering
  \begin{tabular}{|c|c|c|c|}
  \hline
$F_{N}$   & $FI_{N}$ & $FM_{N}$ & $FIM_{N}$ \\
\hline
$1$ & $0$ & $1$ &       $1$   \\
\hline
  \end{tabular}
  \caption{Grundy values $N =2k+1$}
  \label{tab:oddUG}
\end{table}
\end{lemma}

We find that normal play can be won in one move by comparing the counterfeit coin with a real coin. For mis\`ere play the first move is the same. It results in the terminal position with two coins of opposite destinies.

\section{Acknowledgments}

We are grateful to Games at Dal 2016 for getting us together to start this project and to the PRIMES program for allowing us to finish it.

\end{document}